\documentclass[11pt,letterpaper]{amsart}

\usepackage{fancyhdr}
\usepackage{bm}
\usepackage{hyperref}
\hypersetup{hidelinks}
\usepackage{graphicx} 
\usepackage{nicematrix}
\usepackage{amsthm,amssymb,amsmath}
\usepackage[T1]{fontenc}
\usepackage[utf8]{inputenc}
\usepackage{hyperref}
\usepackage{lipsum}
\usepackage{mathrsfs}
\usepackage{enumitem}
\usepackage{stmaryrd}
\usepackage{setspace}
\usepackage{yfonts}
\usepackage{float}
\usepackage{mathtools}
\usepackage{parskip}
\usepackage[all,cmtip]{xy}
\usepackage{tikz-cd}
\tikzcdset{row sep/normal=50pt, column sep/normal=50pt}

\newtheorem{lemma}{Lemma}[section]

\newtheorem{remark}[lemma]{Remark}
\newtheorem{theorem}[lemma]{Theorem}
\newtheorem{theorem*}{Theorem}
\newtheorem{example}[lemma]{Example}
\newtheorem{example*}[lemma]{Example}

\newtheorem{claim}[lemma]{Claim}

\newtheorem{manualtheoreminner}{Theorem}
\newenvironment{manualtheorem}[1]{%
\renewcommand\themanualtheoreminner{#1}%
  \manualtheoreminner
}{\endmanualtheoreminner}
\newtheorem{manualprobleminner}{Problem}
\newenvironment{manualproblem}[1]{%
  \renewcommand\themanualprobleminner{#1}%
  \manualprobleminner
}{\endmanualprobleminner}
\newtheorem{manualsupplementinner}{Supplement}
\newenvironment{manualsupplement}[1]{%
  \renewcommand\themanualsupplementinner{#1}%
  \manualsupplementinner
}{\endmanualsupplementinner}

\newtheorem{manualcorollaryinner}{Corollary}
\newenvironment{manualcorollary}[1]{%
  \renewcommand\themanualcorollaryinner{#1}%
  \manualcorollaryinner
}{\endmanualcorollaryinner}
\newtheorem{manualclaiminner}{Claim}
\newenvironment{manualclaim}[1]{%
  \renewcommand\themanualclaiminner{#1}%
  \manualclaiminner
}{\endmanualclaiminner}

\usepackage[margin=1.1in]{geometry}
\usepackage{blindtext}

\usepackage[backend=biber, style=numeric, sorting=nyt]{biblatex}
\title{Torelli theorems for Quot schemes of vector bundles on curves}
\author{\small{Ashima Bansal, Supravat Sarkar, Shivam Vats}}
\date{}
\begin{document}

\begin{abstract}
For a vector bundle $E$ on a smooth projective curve $C$, one defines the Quot scheme $Q_d(E,C)$ parametrizing subsheaves of $E$ having length $d$ torsion quotients. If $E_i$ is a vector bundle of rank $r\geq 2$ on a smooth projective curve $C_i$ for $i=1,2$, isomorphisms between $\mathbb{P}_{C_i}(E_i)$  preserving the projective bundle structures induce isomorphisms of $Q_d(E_i,C_i)$, called natural isomorphisms. We show that for $d\geq 2$ all isomorphisms between $Q_d(E_i,C_i)$ are natural, except for a non-natural involution of $Q_2(\mathcal{O}_C^{\oplus r},C)$. This gives a complete description of the automorphism group of $Q_d(E,C)$, generalizing previous works of Biswas-Dhillon-Hurtubise and Gangopadhyay. This also shows that one can reconstruct the curve from the Quot scheme, a Torelli-type theorem. As a key step in our proof, we show that any isomorphism between symmetric powers of smooth projective curves preserving big diagonals is natural, which is interesting in its own right.
\end{abstract}
\maketitle
\markboth
{A. Bansal, S. Sarkar, and S. Vats}
{Torelli theorems for Quot schemes}
\begin{center}
\end{center}
\begin{center}
\end{center}

\section{Introduction}
We work throughout over the field $\mathbb{C}$ of complex numbers. For a smooth projective variety $X$ and a positive integer $m$, one defines $X^{[m]}$ to be the Hilbert scheme of length $m$ subschemes of $X$, and $X^{(m)}$ to be the $m$-th symmetric power of $X$, that is, the quotient of the product $X^m$ under the natural action of the symmetric group $S_m$. The study of automorphisms of $X^{[m]}$ and $X^{(m)}$ is a fruitful area of research. Any isomorphism $h:X\to Y$ of smooth projective varieties induces isomorphisms $h^{[m]}:X^{[m]}\to Y^{[m]}$ and $h^{(m)}:X^{(m)}\to Y^{(m)}$, these are called \textit{natural} isomorphisms. The motivating problem is the following.
\begin{manualproblem}{1}\label{1}
Given smooth projective varieties $X$ and $Y$, and a positive integer $m$, characterize all isomorphisms $X^{[m]}\,\to\, Y^{[m]}$ and $X^{(m)}\,\to\, Y^{(m)}$. Determine under what conditions every such isomorphism is natural.
\end{manualproblem}
On the one hand, when $X\,=\,Y$, this problem asks to describe the automorphisms of $X^{[m]}$ and $X^{(m)}$. On the other hand, solutions to the second part of this problem would give rise to Torelli-type theorems: one can recover $X$ from $X^{[m]}$ or $X^{(m)}$ under certain conditions.

Problem \ref{1} has been extensively studied in the literature when $X$ and $Y$ are surfaces, see \cite{belmans2020automorphisms}, \cite{Og}, \cite{Ha}, \cite{hayashi2015universal}, \cite{Sas}, \cite{Gir}, \cite{boissiere2012note} and \cite{bansal2025automorphisms}. In a few cases Problem \ref{1} or its variants are also studied for higher dimensional varieties, see \cite{Sh}, \cite{Wa}, \cite{bansal2025automorphisms} and \cite[Corollary 4.2]{BansalSarkarVatsSymmetricPower}.

Now consider the case when $X$ and $Y$ are curves. In this case the Hilbert scheme coincides with the symmetric power. As the irregularity of $X^{(m)}$ is the genus of $X$, existence of an isomorphism of $X^{(m)}$ and $Y^{(m)}$ implies that $X$ and $Y$ have the same genus, say $g$. Except for the case $g\,=\,m\,=\,2$, \cite[Theorem A.1]{biswas2017automorphisms} shows $X\cong Y$, so Problem \ref{1} essentially asks to characterize automorphisms of $X^{(m)}$. This problem is completely settled for $g\neq 2$ by \cite[Proposition 1.5]{CataneseCiliberto1993}, \cite{ciliberto1993symmetric}, \cite{martens1965extended}, \cite{Ran1986Martens} and \cite{weil2009beweis}. A concise answer is the following. For $g\,=\,0$, we have $X^{(m)}\cong \mathbb{P}^m$, whose automorphisms are well-known. For $g\neq 0,2$ and $m\,\neq\, g-1$, every automorphism of $X^{(m)}$ is natural. For $g\geq 3$ and $X$ hyperelliptic, every automorphism of $X^{(g-1)}$ is also natural. Finally, for $g\,\geq\, 3$ and $X$ non-hyperelliptic, $X^{(g-1)}$ has a non-natural involution, such that every automorphism of $X^{(g-1)}$ is either natural or a composition of that involution with a natural automorphism. For $g=2$, there are non-natural automorphisms in some cases, and the problem is not completely settled. In fact, for $g\,=\,m\,=\,2$, there are non-isomorphic curves with isomorphic second symmetric powers.

For a smooth projective curve $X$, one has the \textit{big diagonal} in $X^{(m)}$, parametrizing the non-reduced divisors. Any natural isomorphism $X^{(m)}\,\to\, Y^{(m)}$ must preserve the big diagonals, in the sense that it maps the big diagonal of $X^{(m)}$ onto the  big diagonal of $Y^{(m)}$. Our first result proves the converse of this, hence completely settles Problem \ref{1} under the assumption of preservation of big diagonals.
\begin{manualtheorem}{A}\label{A}
Let \(C_1\) and \(C_2\) be smooth projective curves and \(d\) a positive integer. Suppose $\varphi\, \colon\, C_1^{(d)} \,\longrightarrow\, C_2^{(d)}$
is an isomorphism preserving the big diagonals. Then \(\varphi\) is natural.
\end{manualtheorem}

One can similarly define big diagonals in $X^{[m]}$ and $X^{(m)}$ for any dimensional $X$. As the big diagonal is the singular locus of $X^{(m)}$ for $\dim X\geq 2$, preservation of big diagonal is automatic when studying Problem \ref{1} for symmetric powers of surfaces and higher dimensional varieties. In the study of Problem \ref{1} for Hilbert schemes of points on surfaces and higher dimensional varieties, preservation of big diagonals was a central assumption in several works,  \cite{bansal2025automorphisms}, \cite{boissiere2012note} and \cite{hayashi2015universal} for example.

For a smooth projective curve $C$, the symmetric power $C^{(m)}$ is an instance of the more general object: Quot scheme $Q_d(E,C)$ of degree $d$ torsion quotients of a vector bundle $E$ on $C$, see \S 2.4 for more details. Given vector bundles $E_i$ on smooth projective curves $C_i$ for $i=1,2$, one defines the notion of \textit{natural} isomorphisms $Q_d(E_1,C_1)\to Q_d(E_2,C_2)$ as follows. Suppose we are given an isomorphism $h:C_1\to C_2$, a line bundle $L$ on $C_1$ and an isomorphism of vector bundles $E_1\to h^*E_2\otimes L$ on $C_1$. As is well-known (see \cite[Exercise II.7.10(d)]{hartshorne2013algebraic}), this is the same data as that of an isomorphism $\varphi:\mathbb{P}_{C_1}(E_1)\to \mathbb{P}_{C_2}(E_2)$ preserving the projective bundle structures, in the sense that $\varphi$ descends to an isomorphism $C_1\to C_2$. Given this data, pullback by $h$ and the natural identification $Q_d(h^*E_2\otimes L,C_1)=Q_d(h^*E_2,C_1)$ induces an isomorphism $Q_d(E_1,C_1)\to Q_d(E_2,C_2)$. This is called a natural isomorphism.

So, in the context of Quot schemes, the analogue of Problem \ref{1} is the following.
\begin{manualproblem}{2}\label{2}
   Given vector bundles $E_i$ of rank $r\geq 2$ on smooth projective curves $C_i$ for $i=1,2$, characterize all isomorphisms $Q_d(E_1,C_1)\to Q_d(E_2,C_2)$. Determine under what conditions every such isomorphism is natural.
\end{manualproblem}

One of the goals of this paper is to completely settle Problem \ref{2}. Our result is the following.
\begin{manualtheorem}{B}\label{B}
Let \(C_1\) and \(C_2\) be smooth projective curves, \(E_i\) a vector bundle on \(C_i\) of rank \(r\geq 2\), and \(d\geq 2\) an integer. Suppose $\sigma\colon Q_d(E_1,C_1)\longrightarrow Q_d(E_2,C_2)$
is an isomorphism. Then one of the following holds:
\begin{enumerate}
     \item \(\sigma\) is natural,
    
    \item \(d=2\), each \(E_i\) is trivial up to line bundle twist, and $\sigma\circ \alpha_{C_1,r}$
    is natural.
\end{enumerate}
\end{manualtheorem}
Here $\alpha_{C_1,r}$ is a non-natural involution of $Q_2(\mathcal{O}_{C_1}^{\oplus r},C_1),$ see \S 2.7 for its definition. We emphasize that here we are not making any assumption like preservation of big diagonals as in Theorem \ref{A}. We also characterize when $Q_{d_i}(E_i,C_i)$'s can be isomorphic for $d_1\neq d_2.$
\begin{manualsupplement}{B}\label{suppB}
Let $d_i$ and $r_i$ be positive integers and $C_i$ a smooth projective curve for $i=1,2$. Suppose $d_1<d_2$. Then the following are equivalent:

\begin{enumerate}
    \item There are vector bundles $E_i$ on $C_i$ of rank $r_i$ with $Q_{d_1}(E_1,C_1)\cong Q_{d_2}(E_2,C_2).$

    \item $d_1=r_2=1$, $d_2=r_1$, and $C_1\cong C_2$ are elliptic curves.
\end{enumerate}

\end{manualsupplement}
Note that Theorem \ref{B} gives a Torelli-type theorem, stating that $C$ can be reconstructed from $Q_d(E,C)$: if rank $E_i\geq 2$ and $Q_d(E_1,C_1)\cong Q_d(E_2,C_2)$, then $C_1\cong C_2$. A special case of this was proven in \cite[Theorem 1.2]{BiswasDhillonHurtubise2015}.

As an application of Theorem \ref{B}, we can completely
describe the automorphisms of $Q_d(E,C)$ in terms of the automorphisms
of $\mathbb{P}_C(E)$. We introduce the following notation to state the
result. For a smooth projective variety $X$, we denote its automorphism group by $\operatorname{Aut}(X)$ and the identity component of the automorphism group scheme by $\operatorname{Aut}^{0}(X)$. For a vector bundle $E$ of rank $r \geq 2$ on a smooth projective
curve $C$, let $\operatorname{Aut}_{\ast}(\mathbb{P}_C(E))$ be the
group of automorphisms of $\mathbb{P}_C(E)$ preserving the projective
bundle structure, that is, which descend to automorphisms of $C$.
Since any map $\mathbb{P}^{r-1}\to C$ is constant except if
$r=2$ and $C\cong\mathbb{P}^1$, one easily sees that $\operatorname{Aut}_{\ast}(\mathbb{P}_C(E))
=
\operatorname{Aut}(\mathbb{P}_C(E)),$
except when $\mathbb{P}_C(E)\cong\mathbb{P}^1\times\mathbb{P}^1$.

As we described, $\operatorname{Aut}_{\ast}(\mathbb{P}_C(E))$ embeds
in $\operatorname{Aut}(Q_d(E,C))$ as the subgroup of natural
automorphisms. Our result is the following.

\begin{manualcorollary}{B}\label{aut}

Let $E$ be a vector bundle of rank $r\geq 2$ on a smooth projective
curve $C$, and $d\geq 2$ an integer. Then the following hold:
\begin{enumerate}
    \item If $d=2$ and $E$ is trivial up to a line bundle twist, then $
    \operatorname{Aut}(Q_d(E,C))$ is a direct product   $\operatorname{Aut}_{\ast}(\mathbb{P}_C(E))
    \times \mathbb{Z}/2\mathbb{Z},$
    where the $\mathbb{Z}/2\mathbb{Z}$-factor is generated by
    $\alpha_{C,r}$.

    \item In all other cases, $    \operatorname{Aut}(Q_d(E,C))
    =
    \operatorname{Aut}_{\ast}(\mathbb{P}_C(E)).$
\end{enumerate}
In particular, we always have $\operatorname{Aut}^{0}(Q_d(E,C))
=
\operatorname{Aut}^{0}(\mathbb{P}_C(E)).$
\end{manualcorollary}

The last statement of Corollary \ref{aut} was studied in the literature and was proved in several special cases. \cite[Theorem 1.1]{BiswasDhillonHurtubise2015} proves it when $E$ is trivial and $C$ has genus $\geq 2$. \cite{Gangopadhyay2019} proves it unconditionally for $r\geq 3$, and for $r=2$ under the additional hypothesis that $E$ is semistable and $C$ has genus $\geq 2$. So our result generalizes these previous works.

If $E$ is a nonzero vector bundle on an elliptic curve $C$, then $C$ is the Albanese variety of $Q_d(E, C)$, see Lemma \ref{smooth albanese}. Analogous to Kummer-type varieties, one can define the Kummer-Quot scheme to be the fibre over $0$ of the Albanese map of $Q_d(E, C)$. Similar proofs as in Theorems \ref{A}, \ref{B} give Torelli-type theorems and naturality of isomorphisms of Kummer-Quot schemes. In \S 5 we state these results and outline the proofs.
\section{Preliminaries}
The following notation will be used:
\begin{itemize}
    \item For a coherent sheaf $\mathcal{F}$ on a variety $X$, and a point $x\in X$, $\mathcal{F}(x)=\mathcal{F}\otimes k(x)$ denotes the fibre of $\mathcal{F}$ over $x$, a finite dimensional $k(x)$-vector space. For a map $\mathcal{F}\xrightarrow{\varphi}\mathcal{G}$ of coherent sheaves on $X$, we thus get an induced $k(x)$-linear map $\mathcal{F}(x)\xrightarrow{\varphi(x)}\mathcal{G}(x)$.
    \item For a smooth variety $X$, its tangent bundle will be denoted by $T_X$. For a morphism $f:X\to Y$ of varieties, the kernel of $df:T_X\to f^*T_Y$, will be denoted by $T_f.$ This is the so-called relative tangent bundle when $f$ is smooth.
    \item For a degree $d$ effective divisor $D$ on a smooth projective curve $C$, sometimes to avoid confusion we will denote the corresponding point of $C^{(d)}$ by $[D]$.
    \item The genus of a smooth projective curve $C$ will be denoted by $g(C).$
    \item A proper morphism between normal varieties $f:X\to Y$ is called a \textit{contraction} if $f_*\mathcal{O}_X=\mathcal{O}_Y$. This is equivalent to the assertion that $f$ has connected fibres.
    \item For a smooth projective variety $X$, we denote its $i$-th Betti number by $b_i(X)$.
    \item 
Let $C$ be an elliptic curve. A vector bundle $E$ on $C$ is said to be \textit{semihomogeneous} if for all $x\in C$, there is a line bundle $L_x$ on $C$ such that $t_x^*E\cong E\otimes L_x,$
where $t_x$ is the translation by $x$. This is equivalent to the assertion that the natural map $\operatorname{Aut}^0(\mathbb P_C(E))
\longrightarrow
\operatorname{Aut}^0(C)$
is surjective.

\end{itemize}
In the following subsections, we define some notions and prove some lemmas which will be used in the proofs.
\subsection{Sheaf of Lie algebras}
 Let \(R\) be a sheaf of commutative rings on a topological space \(X\). An $R$-module $F$ is called a sheaf of Lie algebras if \(F\) is endowed with an \(R\)-bilinear map $F\otimes_R F \longrightarrow F,$
denoted by $s\otimes t\mapsto
[s,t]$
on local sections, such that \([\,\cdot,\cdot\,]\) satisfies the properties of the Lie bracket. Equivalently, for every open set \(U\subseteq X\), \(F(U)\) together with the bracket, forms a Lie algebra. The bracket is itself called the Lie bracket.

The \textit{centroid} of \(F\) is the subsheaf $
\operatorname{Cent}_R(F)\subseteq \operatorname{\mathcal{E}nd}_R(F)$
consisting of local sections \(\varphi\) of \(\operatorname{\mathcal{E}nd}_R(F)\) satisfying
\[
\varphi([a,b])=[\varphi(a),b]=[a,\varphi(b)]
\]
for all local sections \(a\) and \(b\) of \(F\). Note that
\(\operatorname{Cent}_R(F)\) is an \(R\)-algebra, possibly noncommutative.

If \(X\) is a variety, unless otherwise stated, when we say that
``\(F\) is a sheaf of Lie algebras on \(X\)'', it will be assumed that
\(R=\mathcal{O}_X\) and \(F\) is a coherent \(\mathcal{O}_X\)-module.
Similarly, \(\operatorname{Cent}(F)\) will mean
\(\operatorname{Cent}_{\mathcal{O}_X}(F)\). An isomorphism of sheaves of
Lie algebras $F_1\longrightarrow F_2$
will mean an isomorphism of \(\mathcal{O}_X\)-modules preserving the Lie
brackets.

A sheaf $F$ of Lie algebras on a variety $X$  is called \emph{perfect} if the bracket map $F\otimes_{\mathcal{O}_X} F \longrightarrow F$
is surjective. In that case, one easily sees that \(\operatorname{Cent}(F)\) is commutative, \(F\) has a \(\operatorname{Cent}(F)\)-module structure, and the Lie bracket is \(\operatorname{Cent}(F)\)-bilinear (see Section 2.1 of \cite[Section 2.1]{BenkartNeher}).

\begin{example}
If \(E\) is a nonzero vector bundle on a variety \(X\), then
\(\operatorname{\mathcal{E}nd}(E)\) is a sheaf of Lie algebras under the Lie bracket. As in \cite[Section 3]{BiswasGangopadhyaySebastian}, let
\[
\operatorname{ad}(E)
:=
\operatorname{\mathcal{E}nd}(E)/\mathcal{O}_X,
\]
a sheaf of Lie algebras on $X$  isomorphic to $\mathfrak{sl}(E)$, the sheaf of
Lie algebras consisting of trace-zero endomorphisms. We have a direct sum decomposition $\operatorname{\mathcal{E}nd}(E)\cong \mathcal{O}_X\oplus \operatorname{ad}(E)$, where $\mathcal{O}_X$ has zero Lie bracket, see \cite[Page 4]{Sarkar2026Automorphisms}.
    
\end{example}
\subsection{The adjoint Lie algebra}
For a vector space $V$ of dimension $r\geq 2$, the Lie algebra of $\operatorname{Aut}(\mathbb{P}(V))=\operatorname{PGL}(V)$
is $\operatorname{ad}(V)
:=
\frac{\operatorname{End}(V)}
     {\mathbb{C}\cdot\operatorname{id}},$
and is also isomorphic to $H^0\bigl(\mathbb{P}(V),T_{\mathbb{P}(V)}\bigr).$
 Given an isomorphism $\varphi:\mathbb{P}(V)\longrightarrow\mathbb{P}(W),$
the induced map
\[
H^0\bigl(\mathbb{P}(V),T_{\mathbb{P}(V)}\bigr)
\longrightarrow
H^0\bigl(\mathbb{P}(W),T_{\mathbb{P}(W)}\bigr)
\]
can therefore be regarded as an isomorphism of Lie algebras
\[
\operatorname{ad}_{\varphi}:
\operatorname{ad}(V)\longrightarrow\operatorname{ad}(W).
\]

So, $\varphi\mapsto ad_{\varphi}$ gives a map
\begin{equation}\label{injective}
\operatorname{Iso}\bigl(\mathbb{P}(V),\mathbb{P}(W)\bigr)
\longrightarrow
\operatorname{Iso}\bigl(\operatorname{ad}(V),\operatorname{ad}(W)\bigr),
\end{equation}
where \(\operatorname{Iso}(\mathbb{P}(V),\mathbb{P}(W))\) is the set of all
isomorphisms $\mathbb{P}(V)\longrightarrow\mathbb{P}(W),$
and \(\operatorname{Iso}(\operatorname{ad}(V),\operatorname{ad}(W))\) is the
set of all Lie algebra isomorphisms $\operatorname{ad}(V)\longrightarrow\operatorname{ad}(W).$
The map in \eqref{injective} can be shown to be injective: after identifying \(V\) and \(W\) with \(\mathbb{C}^r\), this boils down to the well-known fact that an $r\times r$ matrix commuting with all $r\times r$ matrices is a scalar matrix. An isomorphism $\operatorname{ad}(V)\longrightarrow\operatorname{ad}(W)$
is called \emph{inner} if it is in the image of this map.
\subsection{Universal divisor over symmetric power and secant bundles}
For a smooth projective curve $C$ and a positive integer $d$, let
\[
\Sigma_d(C) \hookrightarrow C \times C^{(d)}
\]
be the universal divisor, consisting of pairs $(x,D)$ with $x\leq D$. We have projections
\[
\Sigma_d(C) \xrightarrow{q} C
\qquad\text{and}\qquad
\Sigma_d(C) \xrightarrow{p} C^{(d)}.
\] For $d\geq 2$, note that
\[
C\times C^{(d-1)}\cong \Sigma_d(C)
\]
via
\[
(x,D)\longmapsto (x,x+D).
\]
Note that the only nontrivial automorphism of $\Sigma_2(C)\cong C^2$
over $C^{(2)}$ is the automorphism swapping the two factors. For $d\geq 3$, we have the following.

\begin{lemma}\label{noauto}
For $d\geq 3$, there is no nontrivial automorphism of $\Sigma_d(C)$ over
$C^{(d)}$.
\end{lemma}
\begin{proof}
Let $U\subseteq C^{(d)}$
be the complement of the big diagonal, and $\Sigma_U:=\Sigma_d(C)\cap (C\times U).$
It suffices to show that $\Sigma_U$ has no nontrivial automorphism over $U$.

Let
\[
Y=\{(x_1,\ldots,x_d)\in C^d\mid x_i\neq x_j
\text{ for }i\neq j\},
\]
a smooth variety. We have $U=Y/S_d$
under the natural free action of $S_d$ on $Y$ by permuting
the coordinates. Also, $\Sigma_U=Y/S_{d-1},$
where $S_{d-1}\subset S_d$
is the subgroup of permutations fixing $d$. The finite
étale maps
\[
Y\longrightarrow\Sigma_U\longrightarrow U
\]
give natural inclusions
\[
\pi_1(Y)\hookrightarrow\pi_1(\Sigma_U)
\hookrightarrow\pi_1(U),
\]
where $\pi_1$ denotes the topological fundamental group. The group of
automorphisms of $\Sigma_U$ over $U$ is $
\frac{N_{\pi_1(U)}\bigl(\pi_1(\Sigma_U)\bigr)}
     {\pi_1(\Sigma_U)},$
so it suffices to show this quotient is trivial. Here for a subgroup $H$ of a group $G$, the normalizer of $H$ in $G$ is denoted by $N_G(H).$

As \(Y\to U\) is the quotient by the free action of \(S_d\), so the subgroup $\pi_1(Y)$ of $\pi_1(U)$ is normal, and
$\frac{\pi_1(U)}{\pi_1(Y)}\cong S_d.$
Similarly, $\frac{\pi_1(\Sigma_U)}{\pi_1(Y)}\cong S_{d-1}.$
Therefore
\[
\frac{N_{\pi_1(U)}\bigl(\pi_1(\Sigma_U)\bigr)}
     {\pi_1(\Sigma_U)}
\cong
\frac{N_{S_d}(S_{d-1})}{S_{d-1}},
\]
which is trivial for \(d\geq 3\).
\end{proof}

Given a vector bundle $E$ on $C$, we define the secant bundle
\[
E^{[d]}:=p_*q^*E,
\]
which is a vector bundle on $C^{(d)}$, as $p$ is finite flat.

\subsection{Quot scheme and Hilbert-Chow morphism}
For a torsion sheaf $F$ on a smooth projective curve $C$, we define a Weil divisor
\[
\operatorname{div}(F):=\sum_{P\in C}\dim(F_P)P,
\]
where $F_P$ is the stalk of $F$ at $P$. The degree of $\operatorname{div}(F)$ is the degree of $F$, which is same as the length of $F$ or $h^0(C, F).$.

Given a nonzero vector bundle $E$ on a smooth projective curve $C$, and a positive integer $d$, $Q_d(E,C)$ is the Quot scheme parametrizing coherent subsheaves $K$ of $E$ such that $E/K$ is torsion of length $d$. This is a smooth projective variety of dimension $rd$. For a line bundle $L$ on $C$, we have natural identification $Q_d(E\otimes L,C)=Q_d(E,C).$

As examples, we have $Q_d(\mathcal{O}_C,C)=C^{(d)},$
and $Q_1(E,C)=\mathbb{P}_C(E).$

We have the Hilbert--Chow morphism $Q_d(E,C)\xrightarrow{\varphi} C^{(d)}$
sending the subsheaf $K$ to $\operatorname{div}(E/K)$. See \cite{BagnarolFantechiPerroni2020} for more details. The map $\varphi$ is a flat contraction. Over the complement $U$ of the big diagonal in $C^{(d)}$, $\varphi$ is smooth with fibres isomorphic to products of projective spaces, see \cite{GangopadhyaySebastianFundamental}.

\begin{lemma}\label{Singular fibre}
Let $E$ be a vector bundle of rank $r \geq 2$ on a smooth projective curve $C$, and $d$ a positive integer. Let $
\varphi \colon Q_d(E,C) \longrightarrow C^{(d)}$
be the Hilbert-Chow morphism. Then the fibre of $\varphi$ over $z \in C^{(d)}$ is singular if and only if $z$ lies in the big diagonal.
\end{lemma}

\begin{proof}
$(\Rightarrow)$ This follows from \cite[Lemma 3.3]{GangopadhyaySebastianFundamental}.

$(\Leftarrow)$ By \cite[Lemma 6.5]{GangopadhyaySebastianFundamental}, we can assume that $z=d[c]$
for some $c\in C$ and $d\geq 2$. The result now follows from \cite[Theorem 1.1(4)]{ItoPunctualQuot}.
\end{proof}

\begin{remark}
By a computation of the tangent map, one can in fact prove that, for a point
$q\in Q_d(E,C)$ corresponding to a degree $d$ torsion quotient $F$ of $E$,
the rank of the tangent map $d\varphi_q$ is the length of the scheme-theoretic
support of $F$. Thus, $\varphi$ is smooth at $q$ if and only if $
F \cong \mathcal{O}_D$
for a degree $d$ divisor $D$ in $C$. This can be regarded as a converse to
\cite[Lemma 3.3]{GangopadhyaySebastianFundamental}.
\end{remark}
\subsection{Scheme-theoretic support of universal quotient}
Let $E$ be a nonzero vector bundle on a smooth projective curve $C$, $d$ a positive integer, and $\varphi:Q_d(E,C)\to C^{(d)}$ the Hilbert-Chow morphism. Let
\[
\Phi
=
\operatorname{id}\times\varphi
\colon
C\times Q_d(E,C)\longrightarrow C\times C^{(d)}
\]
and
\[
\mathcal{D}=\Phi^{-1}(\Sigma_d(C)),
\]
an effective Cartier divisor in $C\times Q_d(E,C)$. The projection
$\mathcal{D}\to Q_d(E,C)$ is finite flat of degree $d$, as the same is
true for $\Sigma_d(C)\to C^{(d)}$.

Let
\[
\pi\colon C\times Q_d(E,C)\longrightarrow C
\]
be the projection. Thus, we have a surjection $\pi^*E\twoheadrightarrow\mathcal{F},$
where $\mathcal{F}$ is the universal quotient sheaf on
$C\times Q_d(E,C)$.

Though we will need a very special case of the following technical
lemma, we prove it in more generality as this is a useful property to
keep in mind.

\begin{lemma}\label{D-module}

In the above notation, $\mathcal{D}$ is the scheme-theoretic support
of $\mathcal{F}$.
    
\end{lemma}

\begin{proof}
Let $Q=Q_d(E,C),$ $\Sigma=\Sigma_d(C)$, 
$U\subseteq C^{(d)}$ the complement of the big diagonal, and $V=\varphi^{-1}(U)\subseteq Q$. Also let $\mathcal{D}_V=\mathcal{D}\cap(C\times V)$
and $\Sigma_U=\Sigma\cap(C\times U).$

By \cite[Lemma 3.3]{GangopadhyaySebastianFundamental}, $\varphi|_V\colon V\longrightarrow U$
is smooth. Hence $\mathcal{D}_V\xlongrightarrow{\Phi}\Sigma_U$
is smooth.
As $\Sigma_U\to U$ is étale, $\Sigma_U$ is smooth, and hence
$\mathcal{D}_V$ is smooth. As $\Sigma_U$ is dense in $\Sigma$ and $\varphi$ is flat, \cite[0H8F]{stacks-project} shows $\mathcal{D}_V$ is dense in $\mathcal{D}$. Since $C\times Q$ is smooth
and $\mathcal{D}$ is generically reduced, we get $\mathcal{D}$ is reduced.
Thus, $\mathcal{D}$ is the closure of $\mathcal{D}_V$ with reduced
induced structure.

Clearly, the support of $\mathcal{F}$ on $C\times V$ contains
$\mathcal{D}_V$, so the scheme-theoretic support of $\mathcal{F}$
contains $\mathcal{D}$. It remains to show $\mathcal{I}_{\mathcal{D}}\subseteq \operatorname{Ann}\mathcal{F},$
where $\mathcal{I}_{\mathcal{D}}$ is the ideal sheaf of $\mathcal{D}$,
and $\operatorname{Ann}\mathcal{F}$ is the annihilator ideal sheaf of
$\mathcal{F}$. This will follow if we show that
$\mathcal{I}_{\mathcal{D}}$ is the $0$-th Fitting ideal
$\operatorname{Fit}_0(\mathcal{F})$, by
\cite[Tag 07ZA]{stacks-project}.
But this is immediate from the description of $\varphi$ as in
\cite[Section 2.2]{BagnarolFantechiPerroni2020}.
\end{proof}

\subsection{Atiyah bundle} Given a vector bundle $E$ on a smooth variety $X$, we have its
Atiyah bundle $\operatorname{At}(E)$ fitting into the short exact
sequence
\[
0\longrightarrow \operatorname{\mathcal{E}nd}(E)
\longrightarrow \operatorname{At}(E)
\longrightarrow T_X
\longrightarrow 0,
\]
see \cite{atiyah1957complex}. Letting $\operatorname{at}(E)
:=
\operatorname{At}(E)/\mathcal{O}_X,$
we get a short exact sequence
\begin{equation}\label{atiyah}
0\longrightarrow \operatorname{ad}(E)
\longrightarrow \operatorname{at}(E)
\longrightarrow T_X
\longrightarrow 0.
\end{equation}
See \cite[Section 3]{BiswasGangopadhyaySebastian} for more details.
\subsection{A non-natural involution of $Q_2(\mathcal{O}_C^{\oplus r},C)$}
 Let $C$ be a smooth projective curve, $r \geq 2$ an integer,
$Q_2 := Q_2(\mathcal{O}_C^{\oplus r},C)$ and
$\varphi \colon Q_2 \longrightarrow C^{(2)}$ the Hilbert--Chow morphism. We will construct an involution $\alpha=
\alpha_{C,r}$ of $Q_2$
over $C^{(2)}$.

Let
\[
\Phi := \operatorname{id} \times \varphi \colon C \times Q_2
\longrightarrow C \times C^{(2)}
\]
and
\[
\mathcal{D}:=\Phi^{-1}(\Sigma_2(C)).
\]
Since $\Sigma_2(C)\cong C^2$, swapping the two factors gives an involution
$\tau$ of $\Sigma_2(C)$. Its pullback $\widetilde{\tau}$ is an involution
of $\mathcal{D}$.

Let $$\widetilde{\eta}:\mathcal{O}_{C\times Q_2}^{\oplus r}
\twoheadrightarrow \mathcal{F}$$
be the universal surjection. As $\mathcal{F}$ is an
$\mathcal{O}_{\mathcal{D}}$-module by Lemma \ref{D-module}, $\widetilde{\eta}$ is induced
by a surjection of sheaves  over $\mathcal{D}$:
$$\eta:\mathcal{O}_{\mathcal{D}}^{\oplus r}\twoheadrightarrow\mathcal{F}.$$
Pulling back by $\widetilde{\tau}$, we obtain a surjection
\[
\eta':\mathcal{O}_{\mathcal{D}}^{\oplus r}
\twoheadrightarrow \widetilde{\tau}^{*}\mathcal{F},
\]
which induces a surjection
\[
\widetilde{\eta'}:\mathcal{O}_{C\times Q_2}^{\oplus r}
\twoheadrightarrow \widetilde{\tau}^{*}\mathcal{F},
\] where $\widetilde{\tau}^{*}\mathcal{F}$ is regarded as a sheaf on $C\times Q_2$. As $\mathcal{F}$ is flat over $Q_2$, we get $\widetilde{\tau}^{*}\mathcal{F}$ is also flat over $Q_2$.
By the universal property of the Quot scheme, $\widetilde{\eta'}$ gives a
morphism
\[
\alpha\colon Q_2\longrightarrow Q_2.
\]
By a local computation or otherwise, one can easily show that
$\alpha$ is an involution.

Over a point $[D]=[x_1+x_2]\in C^{(2)}$ with $x_1\neq x_2$, the fibre of $\varphi$ is $(\mathbb{P}^{r-1})^2$, and $\alpha$ swaps the two factors. From this, one can easily prove the following.

\begin{lemma}\label{commute}
  In the above notation,  $\alpha$ is non-natural and commutes with all natural automorphisms of $Q_2.$
\end{lemma}
If $E\cong \mathcal{O}_C^{\oplus r}$ up to a line bundle twist, we will regard $\alpha_{C,r}$ as an involution of $Q_2(E,C)$, via the natural identification $Q_2(E,C)=Q_2(\mathcal{O}_C^{\oplus r},C).$
\section{Isomorphisms of symmetric powers}
We prove Theorem \ref{A} in this section.

\textit{Proof of Theorem \ref{A}:}
Since $h^1\bigl(C_i^{(d)},\mathcal O_{C_i^{(d)}}\bigr)=g(C_i),$
we have \(g(C_1)=g(C_2)\), call it \(g\). We consider the cases separately.

\medskip

\noindent\underline{\textbf{Case 1: \(g=0\)}}

In this case, we can assume \(C_1=C_2=\mathbb P^1\). Thus $C_1^{(d)}=C_2^{(d)}=\mathbb P^d,$
and the big diagonal is the projective dual of the degree \(d\) rational normal curve \(X\) in the dual projective space \((\mathbb P^d)^\vee\). Now \(\varphi\) induces an automorphism \(\varphi^\vee\) of \((\mathbb P^d)^\vee\) preserving \(X\). This \(\varphi^\vee\) must be induced by an automorphism of \(\mathbb P^1\), via its \(d\)-uple embedding into \((\mathbb P^d)^\vee\). This shows that \(\varphi\) is natural.

\medskip

\noindent\underline{\textbf{Case 2: \(g=1\)}}

In this case, the Albanese variety of \(C_i^{(d)}\) is the elliptic curve \(C_i\). So we can assume \(C_1\) and \(C_2\) are the same elliptic curve. Now the result follows from \cite[Proposition 1.5]{CataneseCiliberto1993}.

\medskip

\noindent\underline{\textbf{Case 3: \(g\geq 2\)}}

Let \(U_i\subseteq C_i^{(d)}\) be the complement of the big diagonal, and let $V_i\subseteq C_i^d$
be the inverse image of \(U_i\) under the quotient map $q_i\colon C_i^d\longrightarrow C_i^{(d)}.$
Since \(\varphi\) preserves the big diagonal, it restricts to an isomorphism $\varphi_0:U_1\longrightarrow U_2.$
The map $V_i\xrightarrow{q_i}U_i$
is finite étale, so we have a natural inclusion $\pi_1(V_i)\subset \pi_1(U_i).$
This is the same as the inclusion of the pure braid group in the braid group: $P_d(C_i)\subset B_d(C_i).$
By \cite[Theorem 1.5]{An2016}, \(P_d(C_i)\) is a characteristic subgroup of \(B_d(C_i)\). This shows that the isomorphism $\varphi_{0*} \colon \pi_1(U_1) \longrightarrow \pi_1(U_2)$
maps \(\pi_1(V_1)\) onto \(\pi_1(V_2)\). So, \(\varphi_0\) lifts to a morphism $\psi_0 \colon V_1 \longrightarrow V_2.$
The same argument lifts \(\varphi_0^{-1}\), hence \(\psi_0\) is an isomorphism. Since \(V_i\) is the ordered configuration space of \(d\) points in \(C_i\), an application of \cite[Theorem 6.1]{ChenSalter}, shows that \(\varphi\) is natural.
\section{Isomorphisms of Quot schemes}
The goal of this section is to prove Theorem \ref{B}, Supplement \ref{suppB} and Corollary \ref{aut}. First we prove a special case of Theorem \ref{B}, where $C_1=C_2=C$ and $\sigma$ is over $C^{(d)}$.
\begin{theorem}\label{B1}
    Let \(C\) be a smooth projective curve, \(E_1\) and $E_2$ vector bundles on \(C\) of rank \(r\geq 2\), and \(d\geq 2\) an integer. Suppose $\sigma\colon Q_d(E_1,C)\longrightarrow Q_d(E_2,C)$
is an isomorphism over $C^{(d)}$. Then one of the following holds:
\begin{enumerate}
     \item \(\sigma\) is natural,
    
    \item \(d=2\), each \(E_i\) is trivial up to line bundle twist, and $\sigma\circ \alpha_{C,r}$
    is natural.
\end{enumerate}
\end{theorem}
\begin{proof}
For ease of notation, let $Q_i:=Q_d(E_i,C)$ and
$\varphi_i:Q_i\longrightarrow C^{(d)}$
the Hilbert--Chow morphisms, and $T_i:=T_{Q_i}$
for $i=1,2$. Let \(U\subset C^{(d)}\) be the complement of the big diagonal. For
\([D]\in U\), let \(\sigma_D\) denote the restriction of \(\sigma\) over
\([D]\). There is a natural identification
\[
\varphi_i^{-1}([D])=\prod_{x\leq D}\mathbb P(E_i(x))
\] for \([D]\in U\), so $\sigma_D$ can be regarded as an isomorphism of these products of projective spaces.
Let $V_i:=\varphi_i^{-1}(U),$ and $f_i:V_i\longrightarrow U$ the restriction of $\varphi_i.$ So $f_i$ is smooth. Also, let $\Sigma=\Sigma_d(C)$ be the universal divisor in
$C\times C^{(d)}$, with projections
\[
\Sigma \xrightarrow{q} C,
\qquad
\Sigma \xrightarrow{p} C^{(d)}.
\]
Since $\varphi_i$ is a contraction,
pushforward of the natural map $T_i\xrightarrow{d\varphi_i}\varphi_i^*T_{C^{(d)}}$
gives a map $$
(\varphi_i)_*T_i
\xrightarrow{\psi_i}
T_{C^{(d)}}.$$
Let $K_i:=\ker(\psi_i).$ So $K_i=(\varphi_{i})_*T_{\varphi_i}$.

When regarded as modules over the constant sheaf $\underline{\mathbb{C}}$ on $C^{(d)}$, both
$(\varphi_i)_*T_i$ and $T_{C^{(d)}}$ are sheaves of Lie algebras under
the Lie bracket of vector fields. A local calculation shows that the
Lie bracket on $K_i$ is $\mathcal{O}_{C^{(d)}}$-bilinear, so $K_i$ is a sheaf of Lie algebras when regarded as an
$\mathcal{O}_{C^{(d)}}$-module. Also, it is easy to see that
$(\operatorname{ad} E_i)^{[d]}$ is a sheaf of Lie
algebras on $C^{(d)}$.

For \([D]\in U\), we have natural direct sum
decompositions:
\begin{equation}\label{dsum 1}
(\operatorname{ad}E_i)^{[d]}([D])
\cong
\bigoplus_{x\leq D}
q^*(\operatorname{ad}E_i)(x,D-x)
=
\bigoplus_{x\leq D}(\operatorname{ad} E_i)(x),
\end{equation}
and
\begin{equation}\label{dsum 2}
K_i([D])
\cong
\bigoplus_{x\leq D}
H^0\!\left(\mathbb P(E_i(x)),
T_{\mathbb P(E_i(x))}\right).
\end{equation}

\begin{claim}\label{Ki}
$K_i \cong (\operatorname{ad} E_i)^{[d]}$
of sheaves of Lie algebras on $C^{(d)}$.
\end{claim}
\begin{proof}
By \cite[Theorem 1.1]{BiswasGangopadhyaySebastian}, we have a natural isomorphism
\[
h:(\varphi_{i})_*T_i \longrightarrow (\operatorname{at}E_i)^{[d]}.
\]
So, by \eqref{atiyah}, there is a natural map $\varphi_{i*}T_i \xrightarrow{t_i}  T_C^{[d]}$
with kernel mapped onto \((\operatorname{ad}E_i)^{[d]}\) via \(h\).

First we show \(\ker t_i=K_i\). This will follow if we have an
injection $T_C^{[d]}
\xrightarrow{\delta_i}
T_{C^{(d)}}$
such that $\delta_i\circ t_i=\psi_i.$
We now define $\delta_i$. Let $D$ be an effective divisor on $C$ of degree $d$ with ideal sheaf $\mathcal{I}_D$. We have
\[
T_{[D]}C^{(d)}= \operatorname{Hom}_D(\mathcal{I}_D|_D,\mathcal{O}_D).
\]
Dualizing the exact sequence
\[
\mathcal{I}_D/\mathcal{I}_D^2
\longrightarrow \Omega_C|_D
\longrightarrow \Omega_D
\longrightarrow 0
\]
and taking global sections, we get an exact sequence
\[
0\longrightarrow
\operatorname{Hom}_D(\Omega_D,\mathcal{O}_D)
\longrightarrow
H^0(D,T_C|_D)
\longrightarrow
T_{[D]}C^{(d)}.
\]
Since $p$ is finite flat, we also have
\[
H^0(D,T_C|_D)
=
(p_*q^*T_C)([D])
=
T_C^{[d]}([D]).
\]
Thus we have defined a map
\[
T_C^{[d]}([D])
\longrightarrow
T_{C^{(d)}}([D]).
\]
This construction can be carried out in families, giving a morphism
\[
\delta_i:T_C^{[d]}\longrightarrow T_{C^{(d)}}.
\]
The map $\delta_i$ is injective, since it is generically injective: $
\operatorname{Hom}_D(\Omega_D,\mathcal{O}_D)=0$
when $D$ is reduced. One can check that $\delta_i\circ t_i=\psi_i.$ 

Now it remains to show the induced isomorphism $\overline{h}:K_i \longrightarrow (\operatorname{ad}E_i)^{[d]}$
preserves Lie brackets. This can be checked on each fibre over \([D]\in U\), which follows as \(\overline{h}([D])\) is induced from the Lie algebra isomorphisms
\[
\operatorname{ad}E_i(x)\xrightarrow{\sim}
H^0\bigl(\mathbb{P}(E_i(x)),T_{\mathbb{P}(E_i(x))}\bigr)
\]
via \eqref{dsum 1} and \eqref{dsum 2}.
\end{proof}

Note that $\sigma$ induces an isomorphism of sheaves of Lie algebras $K_1\rightarrow K_2.$
Hence, by Claim \ref{Ki}, it induces an isomorphism of sheaves of Lie algebras
\[
\theta:
(\operatorname{ad}E_1)^{[d]}
\rightarrow
(\operatorname{ad}E_2)^{[d]}.
\]

\begin{claim}\label{centroid}
$\operatorname{Cent}\bigl((\operatorname{ad}E_i)^{[d]}\bigr)
\cong p_*\mathcal{O}_{\Sigma}.$
\end{claim}

\begin{proof}
As $q^*(\operatorname{ad}E_i)$ is a sheaf of Lie algebras on $\Sigma$, the
Lie bracket of $(\operatorname{ad}E_i)^{[d]}
=
p_*q^*(\operatorname{ad}E_i)$
is $p_*\mathcal{O}_{\Sigma}$-bilinear. Thus $p_*\mathcal{O}_{\Sigma}
\subseteq
\operatorname{Cent}\bigl((\operatorname{ad}E_i)^{[d]}\bigr).$

As $\operatorname{ad}(\mathbb{C}^r)\cong \mathfrak{sl}(\mathbb{C}^r)$ is perfect for $r\geq 2$, we get $\operatorname{ad}E_i$ is perfect. Hence $q^*(\operatorname{ad}E_i)$ is perfect. As $p$ is an affine map, we get $
(\operatorname{ad}E_i)^{[d]}
=
p_*q^*(\operatorname{ad}E_i)$
is perfect. Therefore $\operatorname{Cent}\bigl((\operatorname{ad}E_i)^{[d]}\bigr)$
is commutative. This shows
\[
\operatorname{Cent}\bigl((\operatorname{ad}E_i)^{[d]}\bigr)
=
\operatorname{Cent}_{p_*\mathcal{O}_{\Sigma}}
\bigl(p_*q^*(\operatorname{ad}E_i)\bigr)
=
p_*\operatorname{Cent}\bigl(q^*(\operatorname{ad}E_i)\bigr),
\] where in the last equality we used the equivalence of categories of
quasi-coherent $\mathcal{O}_{\Sigma}$-modules and quasi-coherent
$p_*\mathcal{O}_{\Sigma}$-modules; see \cite[Exercise II.5.17]{hartshorne2013algebraic}.

Thus, it suffices to show $\operatorname{Cent}\bigl(q^*\operatorname{ad}E_i\bigr)
=\mathcal{O}_{\Sigma}.$
This can be proved locally, by trivializing $E_i$ and using the fact that $\operatorname{Cent}_{\mathbb{C}}(\operatorname{ad}(\mathbb{C}^r))=\mathbb{C},$
which follows from Schur's lemma, since $\operatorname{ad}(\mathbb{C}^r)\cong \mathfrak{sl}(\mathbb{C}^r)$
is a simple Lie algebra.
\end{proof}

By Claim \ref{centroid}, $\theta$ induces an isomorphism of
$\mathcal{O}_{C^{(d)}}$-algebras $\theta_0:p_*\mathcal{O}_{\Sigma}
\xrightarrow{\sim}
p_*\mathcal{O}_{\Sigma}.$
This corresponds to an isomorphism
\[
\beta:\Sigma\xrightarrow{\sim}\Sigma
\]
over $C^{(d)}$.
By Lemma \ref{noauto}, either $\beta=\operatorname{id}$,
or $d=2$. We consider these cases separately.

\medskip

\noindent\underline{\textbf{Case 1: $\beta=\operatorname{id}$.}}

Then $\theta_0=\operatorname{id}$, so $\theta$ is
$p_*\mathcal{O}_{\Sigma}$-linear. Hence $\theta$ is induced by an
isomorphism
\[
q^*\operatorname{ad}E_1
\xrightarrow{\sim}
q^*\operatorname{ad}E_2
\]
of sheaves of Lie algebras on $\Sigma$, which is in turn induced by an
isomorphism
\[
\eta:\operatorname{ad}E_1
\xrightarrow{\sim}
\operatorname{ad}E_2
\]
of sheaves of Lie algebras on $C$. Here we are using the fact that $q$ is a contraction.

\begin{claim}\label{claim 3}
There are isomorphisms
$$
h(x,D):\mathbb P(E_1(x))\longrightarrow\mathbb P(E_2(x))$$
for  $[D]\in U$ and $x\leq D,$
such that the following hold:
\begin{enumerate}
    \item[(i)]
$\sigma_D=\prod_{x\leq D}h(x,D),$
\item[(ii)] $\operatorname{ad}_{h(x,D)}=\eta(x).$
\end{enumerate}
\end{claim}

\begin{proof}

By Claim \ref{Ki}, we have an isomorphism of vector bundles $f_{i*}T_{f_i}\cong
(\operatorname{ad}E_i)^{[d]}|_U.$

Since \(\theta\) is induced by \(\eta\), the isomorphism
\[
\theta([D]):
(\operatorname{ad}E_1)^{[d]}([D])
\longrightarrow
(\operatorname{ad}E_2)^{[d]}([D])
\]
corresponds to $\bigoplus_{x\leq D}\eta(x),$ via \eqref{dsum 1}.
In particular, \(\theta([D])\) respects the direct sum decomposition in \eqref{dsum 1}. So the isomorphism $K_1([D])\longrightarrow K_2([D])$
induced by \(\sigma_D\) respects the direct sum decomposition in \eqref{dsum 2}.
Hence \(\sigma_D\) preserves the product structure. In other words, there
are isomorphisms $h(x,D)$
such that $(i)$ holds.

Also, via \eqref{dsum 2}, the isomorphism $K_1([D])\longrightarrow K_2([D])$
induced by \(\sigma_D\) corresponds to $\bigoplus_{x\leq D}\operatorname{ad}_{h(x,D)}.$
We already noted that via \eqref{dsum 1},
\(\theta([D])\) corresponds to $\bigoplus_{x\leq D}\eta(x).$
This proves $(ii)$.
\end{proof}
Now the injectivity of \eqref{injective}
implies that \(h(x,D)\) is independent of \(D\), so there are unique isomorphisms
\[
h(x):\mathbb P(E_1(x))\longrightarrow\mathbb P(E_2(x))
\] for all $x\in C$
satisfying $\operatorname{ad}_{h(x)}=\eta(x).$ These glue together to give an
isomorphism of projective bundles $h:\mathbb P_C(E_1)\xrightarrow{\sim}\mathbb P_C(E_2)$
over \(C\). The natural isomorphism $Q_d(E_1,C)\xrightarrow{\sim}Q_d(E_2,C)$
induced by \(h\) agrees with \(\sigma\) over every
\([D]\in U\), by Claim \ref{claim 3} $(i)$, so agrees everywhere.

\medskip

\noindent
\underline{\textbf{Case 2}: \(d=2\) and \(\beta\neq\operatorname{id}\).}

So $\Sigma\cong C^2$,
and \(\beta\) swaps the two factors. Let $\Sigma\xlongrightarrow{q_i} C$
be the projections onto two factors, with $q=q_1$. Note that \(\theta\) induces a \(p_*\mathcal O_\Sigma\)-linear
isomorphism of sheaves of Lie algebras
\[
p_*q^*\operatorname{ad}(E_1)
   \xrightarrow{\sim}
p_*\beta^*q^*\operatorname{ad}(E_2).
\]
So, we obtain an isomorphism of sheaves of Lie algebras on
\(\Sigma\)
\[
\psi:q_1^*\operatorname{ad}(E_1)
   \xrightarrow{\sim}
q_2^*\operatorname{ad}(E_2).
\]
For \(x,y\in C\), restricting $\psi$ to the fibre over $(x,y)\in C^2\cong\Sigma$
gives an isomorphism of Lie algebras
\[
\psi(x,y):
\operatorname{ad}(E_1(x))
\xrightarrow{\sim}
\operatorname{ad}(E_2(y)).
\]
Similarly as in Claim \ref{claim 3}, one can prove the following claim. Here for \(x\neq y\) in \(C\), the restriction of \(\sigma\) over \([x+y]\in C^{(2)}\) is denoted by \(\sigma_{x+y}\), and \(\varphi_i^{-1}([x+y])\) is identified with \(\mathbb{P}(E_i(x))\times\mathbb{P}(E_i(y))\).
\begin{claim}\label{claim 4}
For each \(x\neq y\) in \(C\), there is an isomorphism
$$\tau(x,y):\mathbb{P}(E_1(x))\to\mathbb{P}(E_2(y))$$
such that

(i) \(\sigma_{x+y}\) is \(\tau(x,y)\times\tau(y,x)\), composed by the swap of the two factors.

(ii) \(\operatorname{ad}_{\tau(x,y)}=\psi(x,y)\).
\end{claim}

In particular, \(\psi(x,y)\) is an inner isomorphism for all \(x\neq y\). As the subgroup of inner automorphisms is closed in the group of all Lie algebra automorphisms of $\operatorname{ad}(\mathbb{C}^r)$, locally trivializing \(E_i\)'s and by a continuity argument we see that \(\psi(x,y)\) is inner for all \((x,y)\in C^2\). Thus one can uniquely define isomorphisms
\(\tau(x,y):\mathbb{P}(E_1(x))\to\mathbb{P}(E_2(y))\)
with $\operatorname{ad}_{\tau_{(x,y)}}=\psi(x,y)$
for all \((x,y)\in C^2\).

Now fix \(y\in C\). The isomorphisms \(\tau(x,y)\) for \(x\in C\) glue to give an isomorphism
\[
\mathbb{P}(E_1)\cong \mathbb{P}(E_2(y))\times C
\] over $C.$
Hence \(E_1\) is trivial up to a line bundle twist. Similarly \(E_2\) is trivial up to a line bundle twist.

Now, replacing \(\sigma\) by \(\sigma\circ\alpha_{C,r}\), we are back to Case 1. This completes the proof.

\end{proof}

\textit{Proof of Theorem \ref{B}:}
Let $\varphi_i\colon Q_d(E_i,C_i)\longrightarrow C_i^{(d)}$
be the Hilbert--Chow morphisms. We first prove the following.

\begin{claim}\label{descend}
\(\sigma\) descends to an isomorphism $\overline{\sigma}\colon C_1^{(d)}\longrightarrow C_2^{(d)}.$
\end{claim}
\begin{proof}
Suppose not. We want to get a contradiction. For simplicity of notation, let $Q_i=Q_d(E_i,C_i).$

As \(\varphi_i\)'s are contractions of relative Picard rank $1$ by
\cite[Theorem 11]{GangopadhyaySebastian}, if \(\sigma\) descends to a morphism $\overline{\sigma}\colon C_1^{(d)}\longrightarrow C_2^{(d)},$
then \(\overline{\sigma}\) is an isomorphism. So, \(\sigma\) does not descend to a morphism \(\overline{\sigma}\). Therefore there must be a curve in
$Q_1$ contracted by \(\varphi_1\), but not contracted by
\(\varphi_2\circ\sigma\). As \(\varphi_1\) has relative Picard rank \(1\), no curve contracted by \(\varphi_1\) is contracted by \(\varphi_2\circ\sigma\). In other
words, for all \(z\in C_1^{(d)}\), the map
\[
\varphi_1^{-1}(z)\xrightarrow{\ \varphi_2\circ\sigma\ } C_2^{(d)}
\]
is finite.

As
\[
\dim \varphi_1^{-1}(z)=d(r-1)
\qquad\text{and}\qquad
\dim C_2^{(d)}=d,
\]
we have $d\geq d(r-1)$,
forcing \(r=2\). Now choosing \(z\in C_1^{(d)}\) general, we have $\varphi_1^{-1}(z)\cong (\mathbb P^1)^d.$
So there is a surjection
\[
(\mathbb P^1)^d\longrightarrow C_2^{(d)}.
\]
Composing it with the Abel-Jacobi map from 
$C_2^{(d)}$ to the  Jacobian of $C_2$
and using the fact that abelian varieties have no rational curves, we get $C_2\cong\mathbb P^1.$
Similarly, \(C_1\cong\mathbb P^1\). So
\[
C_1^{(d)}\cong C_2^{(d)}\cong\mathbb P^d.
\]

Note that $Q_1$ has Picard rank \(2\) by
\cite[Theorem 11]{GangopadhyaySebastian}. So, \(\varphi_1\) and
\(\varphi_2\circ\sigma\) must be the contractions of the two rays of the Mori cone of $Q_1$.
Call these rays \(R_1\) and \(R_2\).

Write $E_1\cong \mathcal{O}_{\mathbb{P}^1}(a)\oplus \mathcal{O}_{\mathbb{P}^1}(b)$ with $a\leq b$. The projection $E_1\twoheadrightarrow \mathcal{O}_{\mathbb{P}^1}(a)$ gives rise to a section $\eta$ of $\varphi$, see \cite[(7), p.~1008]{GangopadhyaySebastian}. Let $S\cong \mathbb{P}^d$ be the image of $\eta$.
As shown in \cite[Proof of Theorem 34]{GangopadhyaySebastian}, pullback of a non-numerically trivial nef line bundle on $Q_1$ by $\eta$ is a line bundle on $C_1^{(d)}\cong \mathbb{P}^d$ that is nef but not ample, hence that pullback is trivial. This means the
class of a curve in \(S\) generates \(R_2\). So, $\varphi_2\sigma(S)$
must be a point. Hence $\sigma(S)\cong\mathbb P^d$
lies in a fibre of \(\varphi_2\).

By \cite[Proposition~6.1]{GangopadhyaySebastianFundamental}, each fibre of \(\varphi_2\)
is irreducible of dimension \(d\). So $\varphi_2^{-1}(w)\cong\mathbb P^d$
for some \(w\in C_2^{(d)}\). But that is impossible: if \(w\) is not in the
big diagonal, then $\varphi_2^{-1}(w)\cong(\mathbb P^1)^d,$
whereas \(\varphi_2^{-1}(w)\) is singular if \(w\) is in the big diagonal,
by Lemma \ref{Singular fibre}. This gives a contradiction.
\end{proof}

Now, by Lemma \ref{Singular fibre}, \(\overline{\sigma}\) must preserve the big diagonals. By Theorem \ref{A}, \(\overline{\sigma}=h^{(d)}\) for some isomorphism $h:C_1\longrightarrow C_2.$
This \(h\) induces a natural isomorphism
\[
Q_d(h^*E_2,C_1)\longrightarrow Q_d(E_2,C_2).
\]
Composing \(\sigma\) with the inverse of this natural isomorphism, we may assume $C_1=C_2=C$ (say)
and that \(\sigma\) is over \(C^{(d)}\). Now Theorem \ref{B} follows from Theorem \ref{B1}.
\qed

\textit{Proof of Supplement \ref{suppB}:}

\underline{$(2)\Rightarrow(1)$:} It follows from the standard fact that $C^{(d)}$ is a projective bundle over $C$ for any elliptic curve $C$ and $d\geq 2$.

\medskip

\underline{$(1)\Rightarrow(2)$:} Let $g_i=g(C_i)$, $Q_i=Q_{d_i}(E_i,C_i)$. Equating dimensions, we get
\begin{equation}\label{dimension}
d_1r_1=d_2r_2.
\end{equation}
So, $r_2<r_1$. If $r_2\geq 2$, then the same (but easier) argument as in the proof of Claim \ref{descend} gives a contradiction. So, $r_2=1$, hence $Q_2\cong C_2^{(d_2)}.$

Suppose $d_1\geq 2$. Using the Poincaré series computed in \cite[Remark 4.6]{BagnarolFantechiPerroni2020}, we get
\begin{equation}\label{betti 1}
b_1\bigl(C_2^{(d_2)}\bigr)=2g_2,\qquad
b_2\bigl(C_2^{(d_2)}\bigr)=\binom{2g_2}{2}+1.
\end{equation}
and
\begin{equation}\label{betti 2}
b_1(Q_1)=2g_1,\qquad
b_2(Q_1)=b_2\bigl(C_1^{(d_1)}\bigr)+1
=\binom{2g_1}{2}+2.
\end{equation}
Here the two equalities in the description of $b_2(Q_1)$ use $r_1\geq 2$ and $d_1\geq 2$, respectively, and in the computation of $b_2\bigl(C_2^{(d_2)}\bigr)$ we used $d_2\geq 2$. From \eqref{betti 1} and \eqref{betti 2}, one immediately gets a contradiction to $Q_1\cong C_2^{(d_2)}.$

So, $d_1=1$. Now $d_2=r_1$ by \eqref{dimension}. So, $C_2^{(d_2)}$ is a projective bundle over $C_1$, hence the Abel--Jacobi map $C_2^{(d_2)}\longrightarrow \operatorname{Jac}(C_2)$
factors through $C_1$, hence has image of dimension $\leq 1$. From this one easily gets $C_1\cong C_2$ and both are elliptic curves.
\qed

\textit{Proof of Corollary \ref{aut}:} Immediate from Theorem \ref{B} and Lemma \ref{commute}.
\qed

\begin{remark}
    For a smooth projective curve $C$ and an integer $r\geq 2$, any map $\mathbb{P}^{r-1} \longrightarrow C$
is constant, except when \(r=2\) and \(C \cong \mathbb{P}^1\). Since $Q_1(E,C)=\mathbb{P}_C(E)$, one easily sees that if $d=1$, then under the remaining hypothesis of Theorem \ref{B}, $\sigma$ is natural, except when $C_i\cong \mathbb{P}^1$ and $E_i\cong \mathcal{O}_{C_i}^2$ up to line bundle twists. In the latter case $Q_i(E_i,C_i)\cong \mathbb{P}^1\times \mathbb{P}^1$, and $\sigma$ might exchange the two factors.
\end{remark}
\begin{remark}
If we have a family of vector bundles on a family of smooth projective curves, one can define the relative Quot scheme as in \cite{Gangopadhyay2019}. Applying Corollary \ref{aut} to each member of the family separately and looking at the relative automorphism group schemes, one can show that \cite[Theorem 2.5(b)]{Gangopadhyay2019} holds without the assumption of semistability and genus made there, and \cite[Theorem 2.5(a)(i)]{Gangopadhyay2019} holds for the full relative automorphism groups, not merely the identity components.
\end{remark}
\section{Torelli theorem for Kummer-Quot scheme}
For an elliptic curve $C$ with origin $O$ and a positive integer $d$, we know that the sum map $C^{(d)}\xlongrightarrow{b} C$ is a $\mathbb{P}^{d-1}$-bundle. We can naturally identify $C$ with the variety $\operatorname{Pic}^d(C)$ parametrizing degree $d$ line bundles on $C$, and $b$ corresponds to the Abel--Jacobi map under this identification. Let $C^{((d))}\cong \mathbb{P}^{d-1}$
be the fibre of $b$ over $O$. The big diagonal in $C^{((d))}$ is defined to be the intersection of the big diagonal of $C^{(d)}$ with $C^{((d))}$.

Suppose $C_i$ is an elliptic curve with origin $O_i$, for $i=1,2$, and $h:C_1\longrightarrow C_2$ is an isomorphism of varieties
such that
\begin{equation}\label{d-torsion}
h^*\mathcal{O}_{C_2}(d\cdot O_2)
\cong
\mathcal{O}_{C_1}(d\cdot O_1)
\end{equation}  This condition is equivalent to saying that $h(O_1)$ is a $d$-torsion point of $C_2$, that is, $h$ preserves the finite subgroups $C_i[d]$ of $d$-torsion points. Such an $h$ induces an isomorphism $C_1^{((d))}
\xrightarrow{\,h^{((d))}\,}
C_2^{((d))}.$
These are called natural isomorphisms. Note that natural isomorphisms preserve big diagonals. Conversely, we can show the following.

\begin{manualtheorem}{A$^{\prime}$}\label{A'}
Let $C_1$ and $C_2$ be elliptic curves, and $d\geq 2$ an integer. Suppose $\varphi:C_1^{((d))}\longrightarrow C_2^{((d))}$
is an isomorphism preserving the big diagonals. Then $\varphi$ is natural.
\end{manualtheorem}

\begin{proof}
The proof is similar to the $g=0$ case of the proof of Theorem \ref{A}. Let $O_i$ be the origin of $C_i$ and $V_i=H^0\bigl(C_i,\mathcal{O}_{C_i}(d\cdot O_i)\bigr).$
The complete linear system $|d\cdot O_i|$ gives a map $C_i\xrightarrow{\,\eta_i\,}\mathbb{P}(V_i).$
We also have a natural identification $
C_i^{((d))}=\mathbb{P}(V_i^*).$

If $d\geq 3$, then $\eta_i$ is an embedding, and the big diagonal in
$C_i^{((d))}$ is the dual variety of $\eta_i(C_i)$. Thus $\varphi$
induces an isomorphism $\psi:\mathbb{P}(V_1)\longrightarrow \mathbb{P}(V_2)$
preserving $\eta_i(C_i)$'s. In other words, there is an isomorphism $h:C_1\longrightarrow C_2$
making the following diagram commute:
\begin{equation}\label{comm diag}
\begin{tikzcd}
C_1 \arrow[r,"h"] \arrow[d,"\eta_1"'] &
C_2 \arrow[d,"\eta_2"] \\
\mathbb{P}(V_1) \arrow[r,"\psi"'] &
\mathbb{P}(V_2).
\end{tikzcd}
\end{equation}
From \eqref{comm diag}, one sees that $h$ satisfies \eqref{d-torsion}. Thus $\psi$, and hence
$\varphi$, is induced by $h$.

For $d=2$, we have $\dim V_i=2$. So lines and hyperplanes in $V_i$
are the same thing, hence we can identify $\mathbb{P}(V_i)=\mathbb{P}(V_i^*).$
In this case, $\eta_i$ is a double cover with branch locus the same
as the big diagonal, which consists of four distinct points. Since a
double cover of $\mathbb{P}^1$ with branch locus four distinct points is
uniquely determined by the branch points, we get isomorphisms $h:C_1\longrightarrow C_2$ and
$\psi:\mathbb{P}(V_1)\longrightarrow\mathbb{P}(V_2)$
making \eqref{comm diag} commute. Now one finishes the proof as in the $d\geq 3$
case.
\end{proof}

Before continuing our discussion, we prove the following Lemma.

\begin{lemma}\label{smooth albanese}
Let $E$ be a nonzero vector bundle on an elliptic curve $C$, and $d$ a positive integer, and $Q=Q_d(E,C)$. Let $Q\xrightarrow{\varphi} C^{(d)}$
be the Hilbert--Chow map and $C^{(d)}\xrightarrow{b} C$
the sum map. Then $a:=b\circ\varphi$ is smooth and is the same as the Albanese map of $Q$. If $E$ is semihomogeneous, then we further have $T_Q\cong \mathcal{O}_Q\oplus T_a.$
\end{lemma}

\begin{proof}
Let $Q\xrightarrow{a_1}A$
be the Albanese map of $Q$. As fibres of $\varphi$ and $b$ are rational, there is $C\xrightarrow{\psi}A$
with $a_1=\psi\circ a$. By the universal property of the Albanese map, there is $A\xrightarrow{\phi}C$
with $\phi\circ a_1=a$. This implies $\phi\circ\psi=\operatorname{id}$. Further, since $A$ is generated as an abelian variety by $a_1(Q)=\psi(C)$, we get $\psi$ is surjective. This forces $\phi$ and $\psi$ to be isomorphisms. Thus $a$ is the same as the Albanese map of $Q$.

First, we prove the Lemma when $E$ is semihomogeneous. If $\operatorname{rank} E=1$, then $Q=C^{(d)}$ so $a=b$ is a $\mathbb{P}^{d-1}$-bundle structure, hence smooth. By \cite[Corollary 9.2]{schroer2009hilbert}, the pullback of $Q\xrightarrow{a}C$
under the "multiplication by $d$"-map $C\xrightarrow{\times d}C$
is projection from a product. Now by \cite[Lemma 6.1]{SarkarTangentHilbert} we have $T_Q\cong \mathcal O_Q\oplus T_a.$ This is the same (easier in this case) proof as in \cite[Proof of (1) and (2), Theorem A]{SarkarTangentHilbert}.  Now assume $\operatorname{rank}E\geq 2$ and $E$ is semihomogeneous. We have a commutative diagram
\[
\begin{tikzcd}
\operatorname{Aut}^0(\mathbb P_C(E))
    \arrow[r,"\cong"]
    \arrow[d]
&
\operatorname{Aut}^0(Q)
    \arrow[d]
\\
\operatorname{Aut}^0(C)
    \arrow[r,"\times d"]
&
\operatorname{Aut}^0(C).
\end{tikzcd}
\]

Here we are identifying $\operatorname{Aut}^0(C)\cong C$ via translations. The top horizontal map is an isomorphism
by Corollary \ref{aut}. The first vertical map is induced by the projection $\mathbb P_C(E)\longrightarrow C,$
and is surjective as $E$ is semihomogeneous. The second vertical map is
induced by the contraction $a$, and has to be surjective, as seen from the diagram. Thus
the map on Lie algebras
\[
H^0(Q,T_Q)\longrightarrow H^0(C,T_C)\cong\mathbb C
\]
is nonzero, hence surjective. So there is a section of $T_Q$ mapping to a nowhere vanishing section of $T_C$ under this map. This gives a section of the map $da:T_Q\longrightarrow a^*T_C.$
Hence $a$ is smooth and $T_Q\cong \mathcal O_Q\oplus T_a.$

Now let $E$ be any nonzero vector bundle. By \cite[Lemma 2.7]{GangopadhyaySebastianPicard}, at a point $\beta$ of $Q$ corresponding to an exact sequence
\begin{equation}\label{beta}
0\longrightarrow K\longrightarrow E\longrightarrow F\longrightarrow 0, 
\end{equation} the tangent map
$(da)_\beta$ is the composition
\[
\operatorname{Hom}(K,F)\xlongrightarrow{-\delta}
\operatorname{Ext}^1(F,F)
\xrightarrow{\operatorname{tr}}
H^1(C,\mathcal{O}_C),
\]
where the first map is obtained by applying $\operatorname{Hom}(-,F)$ to \eqref{beta}, and the second map is the trace map. Since $F$ is torsion, it has zero-dimensional support, so $\operatorname{Ext}^1(E,F)=H^1(C,E^*\otimes F)=0.$
So, \eqref{beta} shows $-\delta$ is surjective. Thus,
\begin{equation}\label{equivalence}
(da)_\beta\text{ is surjective}
\Longleftrightarrow
\operatorname{tr}\text{ is surjective}.
\end{equation}
The upshot is that the second statement in \eqref{equivalence} only depends on $F$ and not on $E$. We can choose a semihomogeneous vector bundle $E'$ on $C$ with a surjection $E'\longrightarrow F.$
Indeed, since $F$ is globally generated, we can choose $E'$ to be a trivial bundle. Now let $Q'=Q_d(E',C)$, and $Q'\xrightarrow{a'} C$
the Albanese map.  Now $a'$ is smooth by the last statement of the Lemma, which we have already proved. By \eqref{equivalence}, applied to $Q'$ instead of $Q$, we get $\operatorname{tr}$ is surjective. Thus $(da)_\beta$ is surjective, hence $a$ is smooth at $\beta$.
\end{proof}
Let $E$ be a nonzero vector bundle on an elliptic curve $C$ with origin $O$, and $a:Q_d(E,C)\longrightarrow C$
the map as in Lemma \ref{smooth albanese}.
Define $Q_d'(E,C)$ to be the fibre of $a$ over $O$. This is a smooth projective variety by Lemma \ref{smooth albanese}. We call it a Kummer-Quot scheme, by its analogy with the generalized Kummer hyperkähler manifold. The restriction of the Hilbert--Chow morphism gives a morphism $Q_d'(E,C)\longrightarrow C^{((d))},$
which we also call the Hilbert--Chow morphism.

Let $E_i$ be a nonzero vector bundle on an elliptic curve $C_i$ with origin
$O_i$, for $i=1,2$. Given an isomorphism $h:C_1\longrightarrow C_2$
satisfying \eqref{d-torsion}, a line bundle $L$ on $C_1$, and an isomorphism of vector bundles $E_1\xrightarrow{\sim}h^*E_2\otimes L$
on $C_1$, the restriction of the induced natural isomorphism $Q_d(E_1,C_1)\longrightarrow Q_d(E_2,C_2)$
gives an isomorphism
\[
Q_d'(E_1,C_1)\longrightarrow Q_d'(E_2,C_2),
\]
which we call a natural isomorphism.

Given $r\geq 2$ and an elliptic curve $C$, the restriction of $\alpha_{C,r}$
gives an involution of $Q_2'\bigl(\mathcal{O}_C^{\oplus r},C\bigr)$
over $C^{((2))}$, but this is the natural involution induced by the inverse map $i:C\to C$.
\begin{manualtheorem}{B$^{\prime}$}\label{B'}
Let \(C_1\) and \(C_2\) be elliptic curves, \(E_i\) a vector bundle on \(C_i\) of rank \(r\geq 2\), and \(d\geq 2\) an integer. Suppose $\sigma\colon Q_d'(E_1,C_1)\longrightarrow Q_d'(E_2,C_2)$
is an isomorphism. Then $\sigma$ is natural.
\end{manualtheorem}
\begin{proof}
The proof is very similar to the proof of Theorem \ref{B}. We outline the steps.
First we prove the analogue of Theorem~4.1.

\begin{manualtheorem}{4.1$^{\prime}$}
Let \(C\) be an elliptic curve, \(E_1\) and $E_2$ vector bundles on \(C\) of rank \(r\geq 2\), and \(d\geq 2\) an integer. Suppose $\sigma\colon Q_d'(E_1,C)\longrightarrow Q_d'(E_2,C)$
is an isomorphism over $C^{((d))}$. Then $\sigma$ is natural.
\end{manualtheorem}

\begin{proof}
We use the notations in the Proof of Theorem \ref{B1}. Additionally, let $Q'_i=Q_d'(E_i,C),$  
$\varphi_i':Q'_i\longrightarrow C^{((d))}$
be the Hilbert--Chow morphisms, $T_i'=T_{Q'_i}$, $U'\subset C^{((d))}$ the
complement of the big diagonal, and for $[D]\in U'$, let $\sigma_D$ be the restriction of $\sigma$ over $[D]$. Also, let $\Sigma' \subset C\times C^{((d))}$
be the intersection of $\Sigma_d(C)$ with $C\times C^{((d))}$, and let
\[
\Sigma' \xrightarrow{q'} C,
\qquad
\Sigma' \xrightarrow{p'} C^{((d))}
\]
be the projections. Note that $q'$ is a $\mathbb{P}^{d-2}$-bundle. Let $V_i'=(\varphi_i')^{-1}(U'),$ and
$f_i':V_i'\longrightarrow U'$
be the restriction of $\varphi_i'$.

For a vector bundle $E$ on $C$, define a vector bundle $E^{[[d]]}$ on
$C^{((d))}$ by
\[
E^{[[d]]}=p'_*(q')^*E.
\]
As $p'$ is finite flat, cohomology and base change (\cite[Tag 0D4E]{stacks-project}) shows
\[
E^{[[d]]}\cong E^{[d]}\big|_{C^{((d))}}.
\]

In the following, we omit the subscript $i$ on $\varphi$, $\varphi'$, $Q$, $Q'$, $E$, $K$ and $K'$ for simplicity. Define $K'=\varphi'_*T_{\varphi'}$, the kernel of the natural map $\varphi'_*T_{Q'}\longrightarrow T_{C^{((d))}}$ and $K=\varphi_*T_{\varphi}$ as in the Proof of Theorem \ref{B1}. 

As $a$ and $b$ are smooth, $d\varphi$ induces isomorphism of the normal bundle of $Q'$ and the pullback of the normal bundle of $C^{((d))}$, both of which are trivial. This means $T_{\varphi'}$ is the kernel of $T_Q|_{Q'}
\xrightarrow{\,d\varphi\,}
(\varphi')^*\left(T_{C^{(d)}}\big|_{C^{((d))}}\right).$ Thus $K'$ is the kernel of the natural map $\varphi'_*(T_Q|_{Q'})\to T_{C^{(d)}}\big|_{C^{((d))}}.$ Using \cite[Tag 0D4E]{stacks-project} and
\cite[Thm.~1.1]{BiswasGangopadhyaySebastian}, we have $\varphi_*T_Q\big|_{C^{((d))}}
\cong
\varphi'_* \bigl(T_Q|_{Q'}\bigr).$ Thus $K'$
is the kernel of the natural map $\varphi_*T_Q\big|_{C^{((d))}}
\longrightarrow
T_{C^{(d)}}\big|_{C^{((d))}}.$ 
\begin{claim}
$K'=K\big|_{C^{((d))}}$.
\end{claim}
\begin{proof}

Note that by Claim \ref{Ki} we have short exact sequences
\begin{equation}\label{ses1}
0\longrightarrow K
\longrightarrow \varphi_*T_Q
\longrightarrow T_C^{[d]}
\longrightarrow 0,
\end{equation}
and
\begin{equation}\label{ses2}
0\longrightarrow T_C^{[d]}
\longrightarrow T_{C^{(d)}}
\longrightarrow \mathcal{G}
\longrightarrow 0,
\end{equation}
where $\mathcal{G}$ is a torsion sheaf. It suffices to show the restriction of both sequences to $C^{((d))}$ are exact. As \eqref{ses1} is a sequence of vector bundles,
its restriction to $C^{((d))}$ is exact. To show the restriction of \eqref{ses2} to $C^{((d))}$ is also exact, it suffices to show that the effective Cartier divisor $C^{((d))}$ does not contain
any associated point of $\mathcal{G}$: this implication follows from a computation of Tor groups and using \cite[Tag 00LD]{stacks-project}. 

Note that the support of $\mathcal{G}$ is contained in the big diagonal. Using the locally free resolution of $\mathcal{G}$ as in \eqref{ses2} and \cite[Tag 090U]{stacks-project}, we see that $\mathcal{G}$ is Cohen-Macaulay with support having pure codimension $1$ in $C^{(d)}$. Thus the only associated point of $\mathcal{G}$ is the generic point of the big diagonal of $C^{(d)}$, which is not contained in $C^{((d))}$. This completes the proof.
\end{proof}
Now 
Using Claim \ref{Ki} and restricting to $C^{((d))}$, we obtain 
\begin{manualclaim}{4.2$'$}
    $K_i'\cong (\operatorname{ad}E_i)^{[[d]]}.$
\end{manualclaim}
By a similar proof as in Claim \ref{centroid} we get
\begin{manualclaim}{4.3$'$}
$\operatorname{Cent}\bigl((\operatorname{ad}E_i)^{[[d]]}\bigr)
=
p'_*\mathcal O_{\Sigma'}.$
\end{manualclaim}
Thus we obtain an isomorphism $\Sigma' \xrightarrow{\beta'} \Sigma'$
over $C^{((d))}$. By the same proof as Lemma \ref{noauto}, we get $\beta'=\operatorname{id}$
or $d=2$. Here, in the proof of Lemma \ref{noauto}, one has to replace $Y$ by
\[
Y'=\{(x_1,x_2,\ldots,x_d)\in Y\mid \sum_{i=1}^dx_i=O\text{ in }C\}.
\] Note that $Y'$ is a dense open subset of the set of all $\underline{x}\in C^d$ whose coordinates sum to $O$ in $C$. This latter set is isomorphic to $C^{d-1}$ as any choice of first $d-1$ coordinates uniquely determine the last coordinate.
Thus $Y'$ is isomorphic to an open subscheme of $C^{d-1}$, hence a smooth variety.

Continuing the proof, we proceed similarly as in Theorem \ref{B1}, and analyze the two cases separately. In Case 1, one defines $h(x,D)$ for $[D]\in U'$ and $x\leq D$ as in Claim \ref{claim 3}. The remaining part of Case 1 proceeds as before. A subtle point is the following:  $h(x)$ is defined a priori for all $x\in C$ such that $x\leq D$ for some $[D]\in U'$. When $d\geq 3$, this is satisfied for all $x$, but for $d=2$ this is satisfied if and only if $x$ is not $2$-torsion. But this does not cause any difficulty, as $\eta(x)$ is defined for all $x$, and is inner for all $x$ in a dense open set, hence by continuity for all $x\in C$. Thus we can uniquely define $h(x)$ for all $x\in C$.

The analysis of Case 2 is easier. Replacing $\sigma$ by its composition with the natural isomorphism $Q_d'(i^*E_1,C_1)\to Q_d'(E_1,C_1)$ induced by the inverse map $i:C_1\to C_1$, we can assume $\beta=\operatorname{id}$. Thus we are back to Case 1.
\end{proof}
Now we are ready to prove Theorem \ref{B'}. By the same proof as in Theorem \ref{B}, it suffices to show the following.
\begin{manualclaim}{4.6$^{\prime}$}\label{4.6'}
$\sigma$ descends to an isomorphism $\bar{\sigma}:C_1^{((d))}\longrightarrow C_2^{((d))}.$
\end{manualclaim}
\begin{proof}
Suppose not. We want to get a contradiction. This is similar to the proof of
Claim \ref{descend}. The main ingredient in that proof was the fact that $\varphi_i$
has relative Picard rank $1$. The same proof as in \cite[Theorem 11]{GangopadhyaySebastian}
shows that $\varphi_i'$ also has relative Picard rank $1$. Now since $
\dim C_2^{((d))}=d-1,$
the same proof as Claim \ref{descend} shows $d-1\geq d(r-1).$
But this is impossible as $r\geq 2$.
\end{proof}

This finishes the proof of Theorem~$B^{'}$.
\end{proof}
\begin{manualsupplement}{B$^{\prime}$}
Let $d_i$ and $r_i$ be positive integers and $C_i$ an elliptic curve, for
$i=1,2$. Suppose $d_1<d_2$. Then the following are equivalent:

\begin{enumerate}
    \item There are vector bundles $E_i$ on $C_i$ of rank $r_i$ with $    Q_{d_1}'(E_1,C_1)\cong Q_{d_2}'(E_2,C_2).$
    \item $d_1=r_2=1,$ and $d_2=r_1.$
\end{enumerate}
\end{manualsupplement}
\begin{proof}
\underline{$(2)\Rightarrow(1)$:}
Clear, as both sides are $\mathbb{P}^{d_2-1}$.

\medskip

\underline{$(1)\Rightarrow(2)$:}
Equating dimensions, we get $r_2<r_1.$
If $r_2\geq 2$, the same proof as Claim \ref{4.6'} gives a contradiction.
So, $r_2=1$, hence $Q_{d_2}'(E_2,C_2)\cong\mathbb{P}^{d_2-1}.$ Thus $Q_{d_1}'(E_1,C_1)\cong\mathbb{P}^{d_2-1}.$

Since $r_1\geq2$, $Q_{d_1}'(E_1,C_1)$ has a non-isomorphic contraction to
$\mathbb{P}^{d_1-1}$. This forces $d_1=1$. Now dimension consideration shows $d_2=r_1.$
\end{proof}

\section{Acknowledgements}
We thank János Kollár and Jakub Witaszek for giving valuable feedback.
\printbibliography

\vspace{30pt}
\begin{flushleft}
{\scshape Department of Mathematics, Shiv Nadar University, NH91, Tehsil
Dadri, Greater Noida, Uttar Pradesh 201314, India}.

{\fontfamily{cmtt}\selectfont
\textit{Email address: ashima.bansal@snu.edu.in} }

\vspace{10pt}

{\scshape Department of Mathematics, Fine Hall, Princeton University, Princeton, NJ 08540, USA}.

{\fontfamily{cmtt}\selectfont
\textit{Email address: ss6663@princeton.edu} }

{\fontfamily{cmtt}\selectfont
\textit{Email address: shivamvatsaaa@gmail.com} }

\end{flushleft}
\end{document}